\documentclass[12pt]{article}  
\pdfoutput=1 	
\usepackage{geometry}                		
\usepackage[parfill]{parskip}    		
\usepackage{graphicx}			
								
\usepackage[hyperindex=true,pagebackref,colorlinks=true,pdfpagemode=none,urlcolor=blue,linkcolor=blue,citecolor=blue,pdfstartview=FitH]{hyperref}
\usepackage{amsmath,amsfonts}
\usepackage{color}
\usepackage{amsthm}
\usepackage{boolexpr}
\usepackage{amssymb,latexsym,pstricks,pst-plot}
\usepackage[english]{babel}
\usepackage{pgf}
\usepackage{tikz}
\usetikzlibrary{arrows,automata}
\usepackage[UKenglish]{isodate}
\usepackage[T1]{fontenc}
\usepackage{listings}
\usepackage{xcolor}
\usepackage{bm}
\usepackage{mathabx}
\usepackage{enumitem}
\usepackage{mathtools}
\allowdisplaybreaks

\usepackage{amssymb}

\theoremstyle{plain}

\newtheorem{theorem}{Theorem}

\newtheorem{lemma}[theorem]{Lemma}
\newtheorem{proposition}[theorem]{Proposition}

\newtheoremstyle{case}{}{}{}{}{}{:}{ }{}
\theoremstyle{case}
\newtheorem{case}{Case}

\theoremstyle{definition}

\newtheorem{remark}[theorem]{Remark}

\title{Shifting Zeckendorf and Chung-Graham representations}
\author{Rob Burns}

\begin{document}
\maketitle
\begin{abstract}
We re-prove some results about integers whose Zeckendorf and Chung-Graham representations satisfy certain conditions. We use properties of the shift operator and use the software package {\tt Walnut}.
\end{abstract}

\section{Introduction}
\label{intro}
This paper investigates integers whose representation satisfies certain properties such as containing or avoiding a fixed element. The representations discussed here are the Zeckendorf and Chung-Graham representations. Our proofs use properties of the shift operator to reduce a problem over $\mathbb{N}$ to one over a small number of integers. We then use the software package {\tt Walnut} to solve the resulting problem.

We denote the Fibonacci numbers by $F_i, \,\, i\geq 0$. They are defined by the recurrence
$$
F_0 = 0, \,\, F_1 = 1, \,\, F_{m+1} = F_{m} + F_{m-1} \,\, \text{ for } \,\, m \geq 1
$$
and satisfy the Binet formula
$$
F_m = (\phi^m - \psi^m)/\sqrt{5}
$$
where $\phi$ is the golden ratio and $\psi = -1/\phi$.

The Fibonacci numbers can be used as a basis for representing integers in a variety of different ways. Some examples are provided in the paper \cite{Shallit:2023aa} by Shallit and Shan and in the paper by Gilson et al. \cite{Gilson:2020aa}. The most well known system involving the Fibonacci numbers is the Zeckendorf representation, which was published by Lekkerkerker in 1952 \cite{lekk} and by Zeckendorf in 1972 \cite{zeck}.

\bigskip

\begin{theorem}[Zeckendorf's theorem]
Any positive integer $n$ can be expressed uniquely as a sum of Fibonacci numbers 
\begin{equation}
\label{zthm}
n = \sum_{i \geq 0} a_i F_{i+2}
\end{equation}
with $a_i \in \{0,1\}$ and $a_i a_{i+1} = 0$ for all $i$.
\end{theorem}

\bigskip
The sum given in equation (\ref{zthm}) is denoted $[a_0 a_1 \dots]_F$ and the Zeckendorf representation of an integer $n$ is written $(n)_F = a_0 a_1 a_2 \dots$. For example, $17 = F_2 + F_ 4 + F_7$, so  
$$
(17)_F = 1 0 1 0 0 1 \text{ \, and \, } [1 0 1 0 0 1]_F = 17.
$$

We will use the least significant digit first form for integer representations. We will use exponentiation to indicate concatenation in integer representations. For example, 
$$
(17)_F = 1 0 1 0 0 1 = 1 0 1 0^2 1 = (10)^2 0 1 .
$$

In 1981, Chung and Graham introduced a representation system which only used the even Fibonacci numbers \cite{CG1981}. Following the paper by Chu, Kanji and Vasseur \cite{Chu:2025aa}, we will call this the Chung-Graham decomposition or representation of the integer $n$.

\bigskip

\begin{theorem}[Chung and Graham]
\label{cglemma}
Every non-negative integer $n$ can be uniquely represented as a sum
\begin{equation}
\label{cgthm}
n = \sum_{i \geq 0} a_i F_{i+2} \,\,\, \text{ where } a_i \in \{0, 1, 2 \} \text{ and } \,\, a_i = 0 \text{ if $i$ is odd}
\end{equation}
so that, if $ i$ and $j$ are even with $a_i = a_j = 2$ and $i < j$, then there is some even $k: i < k< j$, for which $a_k = 0$.
\end{theorem}

\bigskip

The sum given in equation (\ref{cgthm}) is denoted $[a_0 a_1 \dots]_{CG}$ and the Chung-Graham representation of an integer $n$ is written $(n)_{CG} = a_0 a_1 a_2 \dots$. 

Our historical survey begins with the 1972 paper by Carlitz, Scoville and Hoggatt.\cite{Carlitz_1972} They provided a characterisation of integers having a Zeckendorf representation in which $F_k$ is the smallest Fibonacci number. The characterisation is in terms of compound Wythoff sequences, which are constructed by composing the sequences $\{  \lfloor \phi n \rfloor : n \geq 0 \}$ and $\{ \lfloor \phi^2 n \rfloor : n \geq 0 \}$, where $\phi$ is the golden ratio and $\lfloor . \rfloor$ is the floor function. Kimberling \cite{Kimberling_1983} proved that the set of integers which avoid $F_2$ in their Zeckendorf representation is equal to the set $\{\lfloor \phi  n \rfloor - 1 : n \geq 1 \}$. Griffiths established a formula for the set of integers containing $F_k$ in their Zeckendorf representations. He also established formulae for the sets of integers containing two particular Fibonacci numbers in their Zeckendorf representations. Dekking \cite{https://doi.org/10.5281/zenodo.10803363} examined the set of integers which have a fixed, but arbitrary, start to their Zeckendorf representation. He characterised these integers as generalised Beatty sequences, which are of the form $\{ a \lfloor n \phi \rfloor + b n + c : n \geq 0 \}$ for integers $a,b$ and $c$. Dekking also characterised integers which have a fixed but arbitrary word starting at a given position in their Zeckendorf representation and calculated the asymptotic density of these sets of integers.

Recently, Chu, Kanji and Vasseur considered similar questions in relation to the Chung-Graham representation.\cite{Chu:2025aa} They gave a formula for the set of integers which have a Chung-Graham representation which avoids both $F_k$ and $2 F_k$. They also determined those integers for which $F_k$ (or $2 F_{2k}$) is the smallest Fibonacci number in their Chung-Graham representation. Bustos et al.\cite{Bustos:2025aa} determined the set of integers containing $F_{2k}$ in both their Zeckendorf and Chung-Graham representations.

\bigskip

\section{Preliminaries}

In this paper we will use the software package {\tt Walnut}. Hamoon Mousavi, who wrote the  {\tt Walnut} program, has provided an introductory article \cite{Mousavi:2016aa}. Further information about {\tt Walnut} can be found in Shallit's book \cite{Shallit:2022}.  {\tt Walnut} has an inbuilt facility to implement the Zeckendorf  numeration system. Other representational systems can be added to  {\tt Walnut's} program using instructions provided in section 7.12 of  \cite{Shallit:2022}. In \cite{Burns:2025aa}, we described how the Chung-Graham system can be implemented in  {\tt Walnut}. In that paper we created an automaton, {\tt fibcg}, which converts between the Zeckendorf and Chung-Graham representations. It is shown in Figure \ref{fibcg}. Its input consists of two integers, which are read in parallel. The first, $u$, is a valid Zeckendorf representation and the second, $x$, is a valid Chung-Graham representation. It accepts the pair if and  only if $u = x$ as integers. We tell {\tt Walnut} that an integer is in least-significant digit first Chung-Graham format by including the signifier  {\tt ?lsd\_cg} within an instruction.

The automata in this paper will accept integer representations in least significant digit first (lsd) order unless otherwise stated. 

For $k \geq 2$, we define sets $A_k$ based on the Zeckendorf representation of an integer. The set $A_k$ consists of those integers $n$ which have Zeckendorf representation $n = \sum_{i = 1}^m F_{k_i}$ with $k = k_1 < k_2 < \dots < k_m$. The sets $A_k$ give a partition of the positive integers. We define the sets $F_{even}$ and $F_{odd}$ by
\begin{equation}
\label{foe}
F_{even} = \bigcup_{k \geq 1} A_{2k}  \,\, \text{  and  } \,\, F_{odd} = \bigcup_{k \geq 1} A_{2k+1}.
\end{equation}

So, if an integer $n$ has Zeckendorf representation $n = \sum_{i = 1}^m F_{k_i}$ where $k_1 < k_2 < \dots < k_m$, then 
$$
n \in F_{odd} \text{   if }  k_1 \text{   is odd and  }   n \in F_{even} \text{  if }  k_1 \text{   is even.}
$$

These sets can be constructed within {\tt Walnut} using regular expressions.
\begin{verbatim}
reg fibeven lsd_fib "(00)*1(0|1)*":
reg fibodd lsd_fib "0(00)*1(0|1)*":
\end{verbatim}

\bigskip

\begin{figure}[htbp]
   \begin{center}
    \includegraphics[width=6in]{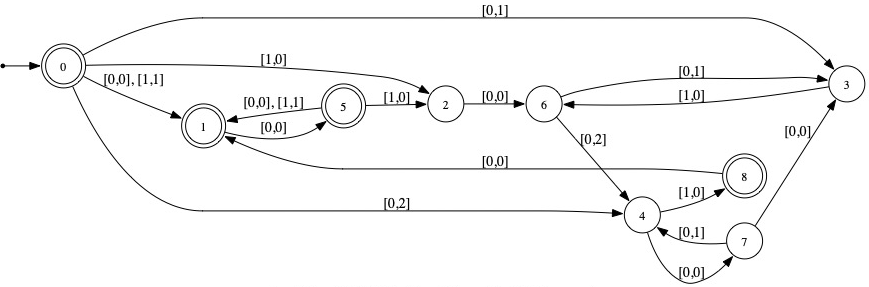}
    \end{center}
    \caption{Automaton which converts between the Zeckendorf and Chung-Graham representations.}
    \label{fibcg}
\end{figure}

We will use a number of automata which appear in theorem 10.11.1 of \cite{Shallit:2022}. Firstly, the sequence  $\{ \lfloor n \phi \rfloor : n \geq 0 \}$ is Fibonacci-synchronised. A least significant digit first synchronising automaton, which we call {\tt phinlsd}, is shown in figure \ref{phinlsd}. The sequence $\{ \lfloor n /\phi \rfloor : n \geq 0 \}$ is also Fibonacci-synchronised. A synchronising automaton for this sequence, which we call {\tt noverphilsd}, is shown in figure \ref{noverphilsd}. 

\bigskip

\begin{figure}[htbp]
   \begin{center}
    \includegraphics[width=6in]{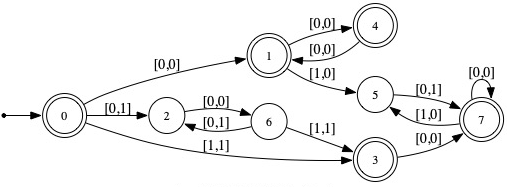}
    \end{center}
    \caption{Synchronised automaton for the function $\lfloor n \phi \rfloor$.}
    \label{phinlsd}
\end{figure}

\bigskip

\begin{figure}[htbp]
   \begin{center}
    \includegraphics[width=6in]{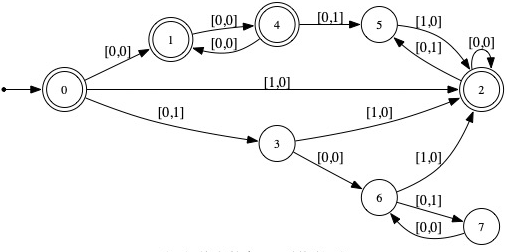}
    \end{center}
    \caption{Synchronised automaton for the function $\lfloor n /\phi \rfloor$.}
    \label{noverphilsd}
\end{figure}

\bigskip

\section{The shift operator}

Since the shift operator plays a significant part in our proofs, we will collect properties of the operator in this section. For least significant digit first integer representations, the shift operator moves a representation to the right and inserts a $0$ at the beginning of the representation. For example, the Zeckendorf representation of the integer $10 = [01001]_F$ becomes $16 = [001001]_F$ under a shift. The single shift of the Chung-Graham representation of an integer does not produce a valid representation because the Chung-Graham system only uses the even-indexed Fibonacci numbers. Shifting two places to the right does produce a valid representation of an integer in the Chung-Graham system. We denote the shift operator on Zeckendorf representations by ${\tt sh_F}$ and on Chung-Graham representations by ${\tt sh_{CG}}$. The shift operator ${\tt sh_F}$ acts on the sets $A_k$ by 
\begin{equation}
\label{shak}
{\tt sh_F} (A_k) = A_{k+1} \,\, \text{  for  } \,\, k \geq 2.
\end{equation}

In {\tt Walnut} the shift operations can be computed from a regular expression.
\begin{verbatim}
reg fibshift {0,1} {0,1} "([0,0]|([1,0][1,1]*[0,1]))*":
reg cgshift {0,1,2} {0,1,2} "([0,0]|([1,0][0,1])|([2,0][0,2]))*":
\end{verbatim}

Shifting two places can then be defined as follows:
\begin{verbatim}
def fibsh2 "?lsd_fib Ey $fibshift(x,y) & $fibshift(y,z)":
def cgsh2 "?lsd_cg Ey $cgshift(x,y) & $cgshift(y,z)":
\end{verbatim}

\bigskip

The following lemma connects the double shift of the Zeckendorf representation with the double shift of the Chung-Graham representation.

\bigskip

\begin{lemma}
\label{fibcgsh}
Shifting an integer $n$ two places to the right in the Zeckendorf and Chung-Graham representations produces the same result if and only if $n \in F_{even}$. In general we have;
\begin{equation*}
{\tt sh_{CG}^2} (n) = [00(n)_{CG}]_{CG} = 
\begin{cases}
{\tt sh_F^2} (n) = [00(n)_{F}]_{F} & \text{   if   }  \,\, n \in F_{even}\\
{\tt sh_F^2} (n)  + 1 = [00(n)_{F}]_{F} + 1  &  \text{   if   } \,\, n \in F_{odd}
\end{cases}
\end{equation*}
\end{lemma}
\begin{proof}
The proof can be completed using {\tt Walnut}. The function {\tt shdiff} measures the difference between shifting two places in the Zeckendorf and Chung-Graham representations.
\begin{verbatim}
def shdiff "?lsd_fib En,r,w,x $fibcg(m,?lsd_cg w) & $fibsh2(m,n) &
     $cgsh2(?lsd_cg w,?lsd_cg x) & $fibcg(r, ?lsd_cg x) & r = n + z":
\end{verbatim}
The following propositions are TRUE, establishing the lemma.
\begin{verbatim} 
eval tcf "?lsd_fib Am ($fibeven(m) <=> $shdiff(m,0))":
eval tcf "?lsd_fib Am ($fibodd(m) <=> $shdiff(m,1))":
\end{verbatim}
\end{proof}

\bigskip

We will need a few identities satisfied by the Fibonacci numbers. In general we have for $m \geq 1$:
\begin{equation}
\label{phif}
F_m \phi  = F_{m+1} - (-\phi)^{-m} \,\, \text{  and  } \,\, F_m \phi^2 = F_{m+2} - (-\phi)^{-m} .
\end{equation}

If $n \geq 0$ is an integer, then (see \cite{Shallit:2022} Theorem 10.11.1),
\begin{align}
\label{shift}
{\tt sh_F} (n)  =& \lfloor  (n+1) \phi \rfloor - 1 \\
{\tt sh_F^2} (n) _F =& \lfloor  (n+1) \phi^2 \rfloor - 2 . 
\end{align}

\bigskip

The following lemmas describe the action on the shift operator on various forms of integers. 

\bigskip

\begin{lemma}
\label{sh1}
Let $x = (n+1) F_m + F_{m-1} \lfloor n/\phi \rfloor$ where $n \geq 0$. Then, if $m > 3$,
$$
{\tt sh_F} (x)  =  (n+1) F_{m+1} + F_{m} \lfloor n/\phi \rfloor.
$$
\end{lemma}
\begin{proof}
From (\ref{shift}):
$$
{\tt sh_F} (x)  = \lfloor  (x+1) \phi \rfloor - 1 
$$
and from (\ref{phif})
\begin{align*}
\lfloor  (x+1) \phi \rfloor &= \lfloor  ((n+1) F_m + F_{m-1} \lfloor n/\phi \rfloor + 1) \phi \rfloor  \\
 &= (n+1) F_{m+1} + F_{m} \lfloor n/\phi \rfloor + \lfloor  \phi - (n+1) (-\phi)^{-m} - \lfloor n/\phi \rfloor (-\phi)^{-m+1}  \rfloor .
\end{align*}

Suppose $\lfloor n/\phi \rfloor = s$. Then $s \phi < n < (s+1) \phi$ and
\begin{align*}
\mid  (n+1) (-\phi)^{-m} & + \lfloor n/\phi \rfloor (-\phi)^{-m+1}  \mid  \, =  \, (n+1) \phi^{-m} - \lfloor n/\phi \rfloor \phi^{-m+1} \\
 & < ((s+1) \phi + 1)\phi^{-m} - s \phi^{-m+1} = \phi^{-m+1} + \phi^{-m} = \phi^{-m+2} .
 \end{align*}
 
If $m$ is even and $\geq 4$, then $(n+1) (-\phi)^{-m}  + \lfloor n/\phi \rfloor (-\phi)^{-m+1} > 0$ and
$$
2 > \phi >  \phi - (n+1) (-\phi)^{-m} - \lfloor n/\phi \rfloor (-\phi)^{-m+1} \geq \phi - \phi^{-2} > 1 .
$$

If $m$ is odd and $\geq 4$, then $(n+1) (-\phi)^{-m}  + \lfloor n/\phi \rfloor (-\phi)^{-m+1} < 0$ and
$$
\phi <  \phi - (n+1) (-\phi)^{-m} - \lfloor n/\phi \rfloor (-\phi)^{-m+1} < \phi + \phi^{-3}  < 2 .
$$
 In both cases 
$$
 \lfloor   \phi - (n+1) (-\phi)^{-m} - \lfloor n/\phi \rfloor (-\phi)^{-m+1}  \rfloor = 1 .
$$
The lemma follows.
\end{proof}

\bigskip

\begin{lemma}
\label{sh2}
Let $x =  \lfloor (n+ \phi^2)/\phi \rfloor F_m + n F_{m+1}$  where $n \geq 0$. Then, if $m \geq 2$,
$$
{\tt sh_F} (x)  =   \lfloor (n+ \phi^2)/\phi \rfloor F_{m+1} + n F_{m+2}.
$$
\end{lemma}
\begin{proof}
From (\ref{shift}):
$$
{\tt sh_F} (x)  = \lfloor  (x+1) \phi \rfloor - 1 
$$
and from (\ref{phif})
\begin{align*}
\lfloor  (x+1) \phi \rfloor &= \lfloor  (\lfloor (n+ \phi^2)/\phi \rfloor F_m + n F_{m+1}  +1) \phi \rfloor \\
 &= \lfloor (n+ \phi^2)/\phi \rfloor F_{m+1} + n F_{m+2}  + \lfloor  \phi - \lfloor (n+ \phi^2)/\phi \rfloor  (-\phi)^{-m} - n (-\phi)^{-m-1}  \rfloor .
\end{align*}
Suppose $\lfloor (n+ \phi^2)/\phi \rfloor = s$. Then $s \phi - \phi^2 < n < (s+1) \phi - \phi^2$ and
\begin{align*}
\mid   \lfloor (n+ \phi^2)/\phi \rfloor  (-\phi)^{-m} + n (-\phi)^{-m-1}  \mid &< s \phi^{-m} - (s \phi - \phi^2) \phi^{-m-1}  \\
 & = \phi^{-m+1} .
\end{align*}
If $m$ is even then
\begin{equation*}
0 \,\,<  \lfloor (n+ \phi^2)/\phi \rfloor  (-\phi)^{-m} + n (-\phi)^{-m-1} < \phi^{-m+1} \leq \phi^{-1}
\end{equation*}
and
$$
2 > \phi  >  \phi - \lfloor (n+ \phi^2)/\phi \rfloor  (-\phi)^{-m} - n (-\phi)^{-m-1}  \geq  \phi - \phi^{-1} = 1 .
$$

If $m$ is odd and $\geq 2$,  then
\begin{equation*}
0 \,\, < \,\, - \lfloor (n+ \phi^2)/\phi \rfloor  (-\phi)^{-m} - n (-\phi)^{-m-1}  \leq \phi^{-2} 
 \end{equation*}
and
$$
1 < \phi  <   \phi - \lfloor (n+ \phi^2)/\phi \rfloor  (-\phi)^{-m} - n (-\phi)^{-m-1}  < \phi + \phi^{-2} = 2 .
$$
The lemma follows since in both cases 
$$
 \lfloor  \phi - \lfloor (n+ \phi^2)/\phi \rfloor  (-\phi)^{-m} - n (-\phi)^{-m-1}  \rfloor = 1 .
$$
\end{proof}

\bigskip

\begin{lemma}
\label{sh3}
Let $x =  \lfloor n \phi \rfloor F_m + n F_{m-1} - F_{m+1}$  where $n \geq 1$. Then, if $m \geq 1$,
$$
{\tt sh_F} (x)  =  \lfloor n \phi \rfloor F_{m+1} + n F_{m} - F_{m+2}
$$
\end{lemma}
\begin{proof}
Use the same approach as for lemmas \ref{sh1} and \ref{sh2}.
\end{proof}

\bigskip

\begin{lemma}
\label{sh4}
Let $x =  (n+1) F_{m} + \lfloor n \phi \rfloor F_{m+1}$  where $n \geq 0$. Then, if $m \geq 2$,
$$
{\tt sh_F} (x)  =  (n+1) F_{m+1} + \lfloor n \phi \rfloor F_{m+2}.
$$
\end{lemma}
\begin{proof}
Use the same approach as for lemmas \ref{sh1} and \ref{sh2}.
\end{proof}

\bigskip

\section{Zeckendorf representations}
\label{zecksec}

We start with a theorem that was published by Kimberling in 1983. It characterises the integers which have a Zeckendorf representation not containing $F_2$.

\bigskip

\begin{theorem}[Kimberling \cite{Kimberling_1983}]
\label{kimb}
$\mathbb{N} \setminus A_2 = \{ \lfloor n \phi \rfloor - 1 : n \geq 2 \}$.
\end{theorem}
\begin{proof}
Since $\mathbb{N} \setminus A_2$ consists of integers which have a Zeckendorf representation of the form $0 (n)_F$ for some integer $n \geq 1$, the result is a restatement of (\ref{shift}).\end{proof}

\bigskip

In 2014, Griffiths characterised the sets $A_k$.

\bigskip

\begin{theorem}[Griffiths \cite{Griffiths_2014} theorem 3.4]
For $k \geq 2$, 
$$
A_k =  \{  \lfloor (n+ \phi^2)/\phi \rfloor F_k + n F_{k+1} : n \geq 0 \}.
$$
\end{theorem}
\begin{proof}
In view of (\ref{shak}) and lemma \ref{sh2}, we only need to establish the result for $k = 2$. This can be done with {\tt Walnut}. An automaton which accepts elements of the set $A_2$ can be constructed using a regular expression
\begin{verbatim}
reg a2 lsd_fib "1(0|1)*":
\end{verbatim}
Since $\phi^2 = \phi + 1$,  $\lfloor (n+ \phi^2)/\phi \rfloor F_k + n F_{k+1} =  \lfloor (n+ 1)/\phi \rfloor F_k + F_k + n F_{k+1}$. The result then follows from the following proposition which returns the value TRUE
\begin{verbatim}
eval test "?lsd_fib Ay $a2(y) <=> 
    (En,x $noverphilsd(n+1,x) & y=x+1+2*n)":
\end{verbatim}
\end{proof}

\bigskip

For $k \geq 0$,  we define the set $U_k$ by
$$
U_k = \bigcup_{i \geq k+2} A_i ,
$$
so $U_k$ consists of those integers having a Zeckendorf representation of the form $0^k (n)_F$ for some integer $n$. From (\ref{shak}), we have
\begin{equation}
\label{fbk}
{\tt sh_F} (U_k) = U_{k+1} .
\end{equation}

\bigskip

\begin{lemma}
\label{bklem}
For $k \geq 1$, the set $U_k$ satisfies
$$
U_k =  \{ \lfloor n \phi \rfloor F_k + n F_{k-1} - F_{k+1} : n \geq 1 \} .
$$
\end{lemma}
\begin{proof}
Given lemma \ref{sh3} and (\ref{fbk}), we only need to prove the result for $k=1, 2$. The result for $k=1$ follows from (\ref{shift}). For $k=2$, we define the set $U_2$ using a regular expression
\begin{verbatim}
reg u2 lsd_fib "00(0|1)*":
\end{verbatim}
Then the $k=2$ case follows from the following TRUE proposition:
\begin{verbatim}
eval test "?lsd_fib Ax (En,y $phinlsd(n,y) & x = y+n-2) <=> $u2(x)":
\end{verbatim}
\end{proof}

\bigskip

Using the above lemma, we can extend Kimberling's result in theorem \ref{kimb} to integers whose Zeckendorf representation does not contain $F_k$ for a given $k$.

\bigskip

\begin{theorem}
For $k \geq 2$, the set of integers whose Zeckendorf representation does not contain $F_k$ is given by
$$
\{ j +  \lfloor n \phi \rfloor F_{k-1} + n F_{k-2} - F_{k} : 0 \leq j < F_k,  \,\, n \geq 1 \} .
$$
\end{theorem}
\begin{proof}
Integers whose Zeckendorf representation does not contain $F_k$ are of the form $j + x$ where $0 \leq j < F_k$ and $x \in U_{k-1}$. The result then follows from lemma \ref{bklem}.
\end{proof}

\bigskip
Dekking considered the set of integers whose Zeckendorf representation has a fixed prefix.\cite{https://doi.org/10.5281/zenodo.10803363} If $b = b_0 b_1 \dots b_{m-1}$ is a finite string of $0$'s and $1$'s containing no $11$, we define $R_b$ to be the set of integers with a Zeckendorf representation which starts $ b_0 b_1 \dots b_{m-1}$.

\bigskip

\begin{theorem}[Dekking \cite{https://doi.org/10.5281/zenodo.10803363} Theorem 2 and Proposition 3]
\label{dek}
Let $b = b_0 b_1 \dots b_{m-1}$ be a length m string of $0$'s and $1$'s containing no $11$. If $b_{m-1} = 0$, then $R_b$ is given by
$$
\{ F_m \lfloor n \phi \rfloor + F_{m-1} n + \gamma_b : n \geq 1 \} 
$$
where $\gamma_b = \sum_{i=0}^{m-2} b_i F_{i+2} - F_{m+1}$.

If $b_{m-1} = 1$, then $R_b$ is given by
$$
\{ F_{m+1} \lfloor n \phi \rfloor + F_{m} n + \gamma_b : n \geq 1 \} 
$$
where $\gamma_b = \sum_{i=0}^{m-1} b_i F_{i+2} - F_{m+2}$.
\end{theorem}
\begin{proof}
We start by noting that it is sufficient to prove the $b_{m-1} = 0$ case. The $b_{m-1} = 1$ case follows by considering the set of integers with Zeckendorf representation starting $b_0 b_1, \dots b_{m-2} 1 0$.

If the length of the word $b$ is $m$, then, since a Zeckendorf representation cannot contain two contiguous $1$'s, we have 
\begin{equation*}
R_b = 
\begin{cases}
\{  \sum_{i=0}^{m-1} b_i F_{i+2}  + x: x \in U_m \} & \text{   if   }  \,\, b_{m-1} = 0 \\
\{  \sum_{i=0}^{m-1} b_i F_{i+2}  + x: x \in U_{m+1} \} & \text{   if   }  \,\, b_{m-1} = 1 .
\end{cases}
\end{equation*}
The result then follows from lemma \ref{bklem}.
\end{proof}

\bigskip

\begin{remark}
The form of $\gamma_b$ in theorem \ref{dek} differs from that which appears in Dekking's theorem. Let $b = b_0 b_1 \dots b_{m-1}$ be a length m string of $0$'s and $1$'s containing no $11$.  Define $T_{00} := \{0 < k < m : b_{k-1} b_k = 0 0 \}$. Then Dekking showed that $\gamma_b = -1 -  \sum_{k \in T_{00}} F_k$. The two expressions are equal as we show below.
\end{remark}

\bigskip

\begin{proposition}
Let $b = b_0 b_1 \dots b_{m-1}$ be a length m string of $0$'s and $1$'s containing no $11$.  Let $T_{00} := \{0 < k < m : b_{k-1} b_k = 0 0 \}$. Then
\begin{equation*}
1 + \sum_{k \in T_{00}} F_k = 
\begin{cases}
-\sum_{i=0}^{m-2} b_i F_{i+2} + F_{m+1} & \text{   if   }  \,\, b_{m-1} = 0 \\
-\sum_{i=0}^{m-1} b_i F_{i+2} + F_{m+2} & \text{   if   }  \,\, b_{m-1} = 1 .
\end{cases}
\end{equation*}
\end{proposition}
\begin{proof}
We will work by induction on the length of the word $b$. The proposition can be verified for small values of $m$ by explicit calculation. Define the integers $X_b$ and $Y_b$ by
$$
X_b := \sum_{k \in T_{b}} F_k \,\, \text{  and   } \,\, Y_b := \sum_{i=0}^{m-1} b_i F_{i+2} .
$$
Then we need to show that 
\begin{equation*}
X_b + Y_b = 
\begin{cases}
F_{m+1} - 1 & \text{   if   }  \,\, b_{m-1} = 0 \\
F_{m+2} - 1 & \text{   if   }  \,\, b_{m-1} = 1 .
\end{cases}
\end{equation*}
Assume the proposition holds for words of length $m$. Let $b = b_0 b_1 \dots b_{m-1}$ and $\bar{b} = b_0 b_1 \dots b_{m-1} b_m$. There are three cases to consider.

\bigskip
\begin{case}[ $b_{m-1} = 1, b_m = 0$]
In this case, $X_{\bar{b}} = X_b$, $Y_{\bar{b}} = Y_b$ and 
$$
X_{\bar{b}} + Y_{\bar{b}}  = X_b + Y_b = F_{m+2} - 1
$$ 
in agreement with the proposition when $ b_m = 0$.
\end{case}
\bigskip
\begin{case}[ $b_{m-1} = 0, b_m = 0$]
In this case, $X_{\bar{b}} = X_b + F_m$, $Y_{\bar{b}} = Y_b$ and $X_b + Y_b = F_{m+1} - 1$. So,  
$$
X_{\bar{b}} + Y_{\bar{b}}  = X_b + Y_b + F_m = F_{m+1} - 1 + F_m = F_{m+2} - 1
$$
in agreement with the proposition when $ b_m = 0$.
\end{case}
\bigskip
\begin{case}[ $b_{m-1} = 0, b_m = 1$]
In this case, $X_{\bar{b}} = X_b$, $Y_{\bar{b}} = Y_b + F_{m+2}$ and $X_b + Y_b = F_{m+1} - 1$. So,  
$$
X_{\bar{b}} + Y_{\bar{b}}  = X_b + Y_b + F_{m+2} = F_{m+1} - 1 + F_{m+2} = F_{m+3} - 1
$$
in agreement with the proposition when $ b_m = 1$.
\end{case}
\end{proof}

\bigskip

\section{Chung-Graham representations}
The results for Zeckendorf representations have recently been extended to Chung-Graham representations by Chu, Kanji and Vasseur\cite{Chu:2025aa} and Bustos et al.\cite{Bustos:2025aa} We will re-prove some of their results in this section. Our method involves converting the problem from the Chung-Graham representation to the Zeckendorf representation using lemma \ref{fibcgsh} and the automaton shown in figure \ref{fibcg}. We can then use the methods from section \ref{zecksec}. For the Chung -Graham representation, we define the sets $B_{2k}$ consisting of those integers $n$ which have a Chung-Graham representation  $n = \sum_{i = 1}^m a_{i} F_{2 k_i}$ with $k = k_1 < k_2 < \dots < k_m$ and each $a_{i} \in \{1, 2 \}$.  These are similar to the sets $A_k$ that were defined for the Zeckendorf representation. The sets $B_{2k}$ give a partition of the positive integers. The Chung-Graham shift operator ${\tt sh_{CG}}$ acts on the sets  $B_{2k}$ by
\begin{equation}
\label{shakcg}
{\tt sh_{CG}^2} (B_{2k}) = B_{2k+2} \,\, \text{  for  } \,\, k \geq 1.
\end{equation}

Since the least significant non-zero digit in a Chung-Graham representation may be either $1$ or $2$, $B_{2k}$ can be partitioned into two subsets, $B_{2k}^{(1)}$ and $B_{2k}^{(2)}$ in an obvious way. The action of the double shift operator on the sets $B_{2k}^{(1)}$ and $B_{2k}^{(2)}$ is similar to (\ref{shakcg}).

Distinct elements of $A_{2k}$ and $B_{2k}^{(1)}$ cannot get too close together.

\bigskip

\begin{lemma}
\label{diff}
If $x \in A_{2k}$ and $y \in B_{2k}^{(1)}$, then $x = y$ or $\mid x - y \mid \, \geq \, F_{2k}$.
\end{lemma}
\begin{proof}
We use {\tt Walnut} for this. We first create automatons {\tt a2k} and {\tt b2k} which accept integers in $A_{2k}$ and $B_{2k}^{(1)}$ respectively.
\begin{verbatim}
reg a2k lsd_fib "(00)*1(0|1)*":
reg b2k lsd_cg "(00)*1(0|1|2)*":
\end{verbatim}

Next we create an automaton to ensure that the same value of $k$ appears in $A_{2k}$ and $B_{2k}^{(1)}$.
\begin{verbatim}
reg samek lsd_fib lsd_cg "[0,0]*[1,1]
    ([0,0]|[0,1]|[0,2]|[1,0]|[1,1]|[1,2])*":
\end{verbatim}

The automaton {\tt mk} measures the value of $k$ appearing in elements of $A_{2k}$:
\begin{verbatim}
reg mk lsd_fib lsd_fib "[0,0]*[1,1]([0,0]|[0,1])*":
\end{verbatim}

We now put together a proposition which selects integers $x \in A_{2k}$ and $y \in B_{2k}^{(1)}$ for the same $k$. It measures the value of $k$ using {\tt mk}, which produces $w = F_{2k}$. It then converts $y$ into an integer $z$ in Zeckendorf form so that it can compare $x$ and $y$ and show that $y-x \geq F_{2k}$ when $y > x$.
\begin{verbatim}
eval test "?lsd_fib Aw,x,y,z ($a2k(x) & $b2k(?lsd_cg y) & 
    $samek(x, ?lsd_cg y) & $mk(w,x) & $fibcg(z, ?lsd_cg y) & x<z) 
    => z>=w+x":
\end{verbatim}

The second proposition does the same, except that it treats the case $x > y$.
\begin{verbatim}
eval test "?lsd_fib Aw,x,y,z ($a2k(x) & $b2k(?lsd_cg y) & 
    $samek(x, ?lsd_cg y) & $mk(w,x) & $fibcg(z, ?lsd_cg y)  & x>z) 
    => x>z+w":
\end{verbatim}
Since both propositions are TRUE, the lemma follows.
\end{proof}

\bigskip

We introduce sequence A276885 from the OEIS\cite{oeis}, which we denote by {\tt sc}. It is described as the sums-complement of the sequence $\{ \lfloor n \phi^2 \rfloor : n \geq 1 \}$ and satisfies the formula ${\tt sc} (n) = 2 \lfloor n/\phi \rfloor + n + 1$ for $n \geq 0$.\cite{Dekking:2017aa}

\bigskip
\begin{theorem}
\label{feunion}
The set $F_{even}$ is the disjoint union of $\mathbb{N} \setminus B_{2}$ and sequence A276885 from the OEIS.
\end{theorem}
\begin{proof}
Sequence A276885 is defined in {\tt Walnut} as
\begin{verbatim}
def sc "?lsd_fib En,y $phinlsd(n,y) & x = 2*y+n+1":
\end{verbatim}

The set $\mathbb{N} \setminus B_{2}$ is defined by:
\begin{verbatim}
reg cg0 lsd_cg "0(0|1|2)*":
\end{verbatim}

The disjointness of $\mathbb{N} \setminus B_{2}$ and sequence A276885 follows from the proposition below, which is FALSE. Remember that the automaton, {\tt fibcg} converts between the Zeckendorf and Chung-Graham representations.
\begin{verbatim}
eval test "?lsd_fib Ex,y $sc(x) & 
    $fibcg(x, ?lsd_cg y) & $cg0(?lsd_cg y)":
\end{verbatim}
The theorem is proved by the following TRUE proposition
\begin{verbatim}
eval test "?lsd_fib Ax (x>0 & $fibeven(x)) <=> 
  (x>0 & ($sc(x)|(Ey $fibcg(x, ?lsd_cg y) & $cg0(?lsd_cg y))))":
\end{verbatim}
\end{proof}

\bigskip
\begin{theorem}
$$
\mathbb{N} \setminus B_{2} \,\, \subseteq \,\,  \{ \lfloor n/\phi \rfloor + n : n \geq 1 \}  .
$$
\end{theorem}
\begin{proof}
The inclusion follows from the {\tt Walnut} proposition
\begin{verbatim}
eval test "?lsd_fib Ax (Ey $fibcg(x, ?lsd_cg y) & $cg0(?lsd_cg y)) 
    => (Em,n $noverphilsd(m,n) & (x=n+m))":
\end{verbatim}
\end{proof}

\bigskip

\begin{theorem}[Chu, Kanji and Vasseur\cite{Chu:2025aa}  Proposition 4.3]
\label{ckv43}
For $k \geq 1$
 $$
 B_{2k} = \{ (n+1) F_{2k} + \lfloor n/\phi \rfloor F_{2k-1} : n \geq 0 \}.
$$
\end{theorem}
\begin{proof}
If $n \not \in B_{2}$, then $n \in F_{even}$ by theorem \ref{feunion}. Hence, from lemma \ref{fibcgsh}, shifting the Chung-Graham representation of $n$ two places to the right produces the same result as shifting the Zeckendorf representation of $n$ two places to the right. By (\ref{shakcg}) and lemma \ref{sh1}, we only need to prove the result for $k=1, 2$. This can be done using {\tt Walnut}. The sets $B_{2}$ and $B_{4}$ can be defined using regular expressions.
\begin{verbatim}
reg b2 lsd_cg "[1|2][0|1|2]*":
reg b4 lsd_cg "00[1|2][0|1|2]*":
\end{verbatim}
\bigskip
The $k=1$ case follows from the proposition below:
\begin{verbatim}
eval test1 "?lsd_fib Ax (En,y $noverphilsd(n,y) & x = n+1+y)  
    <=> (Ez $fibcg(x, ?lsd_cg z) & $b2(?lsd_cg z) )":
\end{verbatim}
The $k=2$ case follows from the proposition below:
\begin{verbatim}
eval test2 "?lsd_fib Ax (En,y $noverphilsd(n,y) & x = 3*n+3+2*y) 
     <=> (Ez $fibcg(x, ?lsd_cg z) & $b4(?lsd_cg z))":
\end{verbatim}
\end{proof}

\bigskip

\begin{theorem}[Chu, Kanji and Vasseur\cite{Chu:2025aa}  Theorem 1.3]
For $k \geq 1$, the set of all positive integers whose Chung-Graham representation does not contain $F_{2k}$ or $2 F_{2k}$ is given by
$$
\{ j : 1 \leq j < F_{2k} \} \,\, \cup \,\, \{ \,\, j + (n+1) F_{2m} + \lfloor n/\phi \rfloor F_{2m-1} : 0 \leq j < F_{2k} , \,\, n \geq 0, \,\, m > k \,\, \} .
$$
\end{theorem}
\begin{proof}
Apart from the integers smaller than $F_{2k}$, integers whose Chung-Graham representation does not contain $F_{2k}$ or $2 F_{2k}$ are of the form $j + x$ where $0 \leq j < F_{2k}$ and $x \in B_{2m}$ for some $m > k$. The result then follows from theorem \ref{ckv43}.
\end{proof}

\bigskip

We now consider which values of $n$ in theorem \ref{ckv43} produce integers in $B_{2k}^{(1)}$ and which values of $n$ produce integers in $B_{2k}^{(2)}$.

\bigskip

\begin{theorem}
When $n = 0$ or $n = \lfloor m \phi \rfloor + 1$ for some integer $m \geq 1$, the integer produced by the formula in theorem \ref{ckv43} is in $B_{2k}^{(1)}$. When $n = \lfloor m \phi^2 \rfloor + 1$ for some integer $m \geq 0$, the integer produced by the formula in theorem \ref{ckv43} is in $B_{2k}^{(2)}$.
\end{theorem}
\begin{proof}
We note that 
$$
\{  \lfloor m*\phi \rfloor  + 1 : m \geq 1 \} \,\, \cup \,\,  \{   \lfloor m \phi^2 \rfloor + 1 : m \geq 0 \} = \mathbb{N}
$$
and the two sets on the left hand side are disjoint. Therefore, we only need to prove the result for $n = \lfloor m \phi \rfloor + 1$. By lemma \ref{sh1}, shifting integers of the form  $(n+1) F_m + F_{m-1} \lfloor n/\phi \rfloor$ does not change the value of $n$. Since $B_{2k} \subseteq F_{even}$ when $k > 1$ by theorem \ref{feunion}, (\ref{shakcg}) and lemma \ref{fibcgsh} imply that it is sufficient to prove the theorem when $k = 1, 2$.
We first define the sets $B_{2}^{(1)}$ and $B_{4}^{(1)}$ using regular expressions.
\begin{verbatim}
reg b21 lsd_cg "1(0|1|2)*":
reg b41 lsd_cg "001(0|1|2)*":
\end{verbatim}
For the $k = 1$ case, the following automaton accepts values of $n$ in theorem \ref{ckv43} that produce integers in $B_{2}^{(1)}$:
\begin{verbatim}
def thm161 "?lsd_fib Ex,y,z $noverphilsd(n,y) & x = (n+1) + y
    & $fibcg(x, ?lsd_cg z) & $b21(?lsd_cg z)":
\end{verbatim}
The following proposition then shows that $n$ produces an integer in $B_{2}^{(1)}$ if and only if $n=0$ or $n = \lfloor m*\phi \rfloor  + 1$ for some integer $m \geq 1$.
\begin{verbatim}
eval test "?lsd_fib An $thm161(n) <=> 
    (n=0|(Em (m>0) & $phinlsd(m,n-1)))":
\end{verbatim}
For the $k = 2$ case, the following automaton accepts values of $n$ in theorem  \ref{ckv43} that produce integers in $B_{4}^{(1)}$:
\begin{verbatim}
def thm162 "?lsd_fib Ex,y,z $noverphilsd(n,y) & x = (3*n+3) + 2*y
    & $fibcg(x, ?lsd_cg z) & $b41(?lsd_cg z)":
\end{verbatim}
The following proposition then shows that $n$ produces an integer in $B_{4}^{(1)}$ if and only if $n=0$ or $n = \lfloor m*\phi \rfloor  + 1$ for some integer $m \geq 1$.
\begin{verbatim}
eval test "?lsd_fib An $thm162(n) <=> 
    (n=0|(Em (m>0) & $phinlsd(m,n-1)))": 
\end{verbatim}
This completes the proof.
\end{proof}

\bigskip

\begin{theorem}[Bustos et al. \cite{Bustos:2025aa} Theorem 1.2]
For $k \geq 1$, the set of positive integers that have $F_{2k}$ in both of their Zeckendorf and Chung-Graham representations is 
$$
\{ \,\, j + n F_{2k} + \lfloor (n-1) \phi \rfloor F_{2k+1} : 0 \leq j < F_{2k-1} , \,\, n > 0  \} .
$$
\end{theorem}
\begin{proof}
Integers $m$ satisfying the theorem can be written $m = j_1 + x = j_2 + y$, where $j_1, j_2 <  F_{2k-1}, x \in A_{2k}$ and $y \in B_{2k}^{(1)}$. So,
$$
0 = j_1 - j_2 + x - y .
$$
However, $\mid j_1 - j_2 \mid < F_{2k-1}$ and either $x = y$ or $\mid x - y \mid \geq F_{2k}$ by lemma \ref{diff}. So, we must have $j_1 = j_2$ and $x = y$. Hence, integers satisfying the theorem are of the form $j + x$, where $0 \leq j < F_{2k-1}$ and $x \in A_{2k} \cap B_{2k}^{(1)}$. We now show that 
\begin{equation}
\label{abinter}
A_{2k} \cap B_{2k}^{(1)}  = \{ \,\, n F_{2k} + \lfloor (n-1) \phi \rfloor F_{2k+1} : n > 0  \} .
\end{equation}
By lemma \ref{fibcgsh}, ${\tt sh_F^2} = {\tt sh_{CG}^2}$ on $A_{2k} \cap B_{2k}^{(1)}$ and so ${\tt sh_F^2} (A_{2k} \cap B_{2k}^{(1)}) = A_{2k+2} \cap B_{2k+2}^{(1)}$. Therefore, applying lemma \ref{sh4} we only need to show (\ref{abinter}) holds for $k=1$. We can do this with {\tt Walnut} using the automata {\tt a2} and {\tt b21} previously defined for the sets $A_{2}$ and $B_{2}^{(1)}$ respectively. The following proposition is TRUE.
\begin{verbatim}
eval test "?lsd_fib Ax (Ey $a2(x) & $fibcg(x,?lsd_cg y) & 
    $b21(?lsd_cg y)) <=> (En,z n>0 & $phinlsd(n-1,z) & x=n+2*z)":
\end{verbatim}
\end{proof}

\bigskip

\bibliographystyle{plain}
\begin{small}
\bibliography{CGZ}

\begin{thebibliography}{10}

\bibitem{oeis}
OEIS Foundation~Inc. (2025).
\newblock The on-line encyclopedia of integer sequences.

\bibitem{Burns:2025aa}
Rob Burns.
\newblock Chung-{G}raham and {Z}eckendorf representations.
\newblock {\em arXiv}, 02 2025.

\bibitem{Bustos:2025aa}
Lucas Bustos, Hung~Viet Chu, Minchae Kim, Uihyeon Lee, Shreya Shankar, and
  Garrett Tresch.
\newblock Integers having ${F}_{2k}$ in both {Z}eckendorf and {C}hung-{G}raham
  decompositions.
\newblock {\em arXiv}, 04 2025.

\bibitem{Carlitz_1972}
L.~Carlitz, Richard Scoville, and V.E. Hoggatt.
\newblock {F}ibonacci representations.
\newblock {\em The Fibonacci Quarterly}, 10(1):1--28, January 1972.

\bibitem{Chu:2025aa}
Hung~Viet Chu, Aney~Manish Kanji, and Zachary~Louis Vasseur.
\newblock Fixed-term decompositions using even-indexed {F}ibonacci numbers.
\newblock {\em arXiv}, 01 2025.

\bibitem{CG1981}
F.~Chung and R.~L. Graham.
\newblock On irregularities of distribution in finite and infinite sets.
\newblock {\em Colloq. Math. Soc. J{\'a}nos Bolyai}, 37:181--222, 1981.

\bibitem{https://doi.org/10.5281/zenodo.10803363}
F.M. Dekking.
\newblock The structure of {Z}eckendorf expansions.
\newblock {\em Integers}, 21, 2021.

\bibitem{Dekking:2017aa}
Michel Dekking.
\newblock The {F}robenius problem for homomorphic embeddings of languages into
  the integers.
\newblock {\em arXiv}, 12 2017.

\bibitem{Gilson:2020aa}
Amelia Gilson, Hadley Killen, Tam{\'a}s Lengyel, Steven~J. Miller, Nadia Razek,
  Joshua~M. Siktar, and Liza Sulkin.
\newblock {Z}eckendorf's theorem using indices in an arithmetic progression.
\newblock {\em arXiv}, 05 2020.

\bibitem{Griffiths_2014}
Martin Griffiths.
\newblock Fixed-term {Z}eckendorf representations.
\newblock {\em The Fibonacci Quarterly}, 52(4):331--335, November 2014.

\bibitem{Kimberling_1983}
Clark Kimberling.
\newblock One-free {Z}eckendorf sums.
\newblock {\em The Fibonacci Quarterly}, 21(1):53--57, February 1983.

\bibitem{lekk}
C.~G. Lekkerkerker.
\newblock {V}oorstelling van natuurlijke getallen door een som van {F}ibonacci.
\newblock {\em Simon Stevin}, 29:190--195, 1952.

\bibitem{Mousavi:2016aa}
Hamoon Mousavi.
\newblock Automatic theorem proving in {{\tt Walnut}}.
\newblock {\em arXiv}, 2016.

\bibitem{Shallit:2022}
Jeffrey Shallit.
\newblock {\em The Logical Approach to Automatic Sequences}.
\newblock Cambridge University Press, Sep 2022.

\bibitem{Shallit:2023aa}
Jeffrey Shallit and Sonja~Linghui Shan.
\newblock A general approach to proving properties of {F}ibonacci
  representations via automata theory.
\newblock In {\em Electronic Proceedings in Theoretical Computer Science},
  volume 386, pages 228--24. Open Publishing Association, 2023.

\bibitem{zeck}
E.~Zeckendorf.
\newblock {R}epr{\'e}sentation des nombres naturels par une somme de nombres de
  {F}ibonacci ou de nombres de {L}ucas.
\newblock {\em Bull. Soc. R. Sci. Li{\`e}ge}, 41:179--182, 1972.

\end{thebibliography}
\end{small}

\end{document}